\documentclass[11pt]{article}
\usepackage{amssymb}
\usepackage{amsthm}
\usepackage{amsmath,amsfonts,amssymb}
\usepackage{graphicx}

\newtheorem{theorem}{Theorem}[section]

\newtheorem{proposition}[theorem]{Proposition}

\newtheorem{definition}[theorem]{Definition}

\begin{document}

\title{Existence and Properties of the State Operator in \\Dynamic User Equilibrium}
\author{Terry L. Friesz$^{a}\thanks{Corresponding author, e-mail: tfriesz@psu.edu}$
\qquad Ke Han$^{b}\thanks{e-mail: kxh323@psu.edu}$
\qquad Tao Yao$^{a}\thanks{e-mail: tyy1@engr.psu.edu}$ \\\\
%EndAName
$^{a}$\textit{Department of  Industrial and Manufacturing Engineering}\\
\textit{Pennsylvania State University, PA 16802, USA}\\
$^{b}$\textit{Department of Mathematics}\\
\textit{Pennsylvania State University, PA 16802, USA}}
\date{}
\maketitle
%\date{\today}

\begin{abstract}
In this paper, we establish and prove analytical properties of the state operator embedded in an optimal control problem, in the context of {\it dynamic user equilibrium} (DUE) models (Friesz et al. 1993).
\end{abstract}

\section{Introduction}
This paper considers a very general differential variational inequality formulation of the dynamic traffic assignment problem and characteristics of the travel delay or latency operator that assure a solution exists. Because the delay operator is the outcome of a separate differential algebraic equation (DAE) model that expresses unit travel time for each state (volume) encountered, it is referred to herein as the ``state operator".

\section{Preliminaries}
\begin{theorem}\label{fpthm} {\bf (Contraction mapping theorem)}
Let $X$ be a Banach space, $\Theta$ a metric space, and let $\Phi: \Theta\times X\rightarrow X$ be a continuous mapping such that, for some $\kappa<1$, 
\begin{equation}\label{fpthmeqn1}
\left\|\Phi(\theta,\,x)-\Phi(\theta,\,y)\right\|~\leq~\kappa \left\|x-y\right\|\qquad \forall\, \theta,\,x,\,y
\end{equation}
Then for each $\theta\in\Theta$ there exists a unique fixed point $x(\theta)\in X$ such that 
\begin{equation}\label{fpthmeqn2}
x(\theta)~=~\Phi\big(\theta,\,x(\theta)\big)
\end{equation}
The map $\theta\mapsto x(\theta)$ is continuous. Moreover, for any $\theta\in\Theta,\,y\in X$ one has
\begin{equation}\label{fpthmeqn3}
\left\|y-x(\theta)\right\|~\leq~{1\over 1-\kappa}\left\|y-\Phi(\theta,\,y)\right\|
\end{equation}
\end{theorem}

\vskip 3em

We begin by considering the control vector
$$
u\in\big(\mathcal{L}^2[t_0,\,t_f]\big)^m
$$
and associated operator
\begin{equation}\label{eqn1}
x(u,\,t)~=~\hbox{arg}\left\{{dy\over dt}=f\big(y(t),\,u(t),\,t\big),\,y(t_0)=x_0,\,\Gamma[y(t_f),\,t_f]=0\right\}\in\big(C^0[t_0,\,t_f]\big)^n
\end{equation}
where
\begin{align}
\label{eqn2}
x_0&~\in~\mathbb{R}^n\\
f&~:~\mathbb{R}^n\times \mathbb{R}^m\times \mathbb{R}^1 ~\longrightarrow~\mathbb{R}^n\\
\label{eqn3}
\Gamma&~:~\mathbb{R}^n\times\mathbb{R}^1 ~\longrightarrow~\mathbb{R}^n
\end{align}
$\big(\mathcal{L}^2[t_0,\,t_f]\big)^m$ is the m-fold product of the space of square-integrable functions $\mathcal{L}^2[t_0,\,t_f]$ with inner product defined by
\begin{equation}\label{eqn4}
\big<u,\,v\big>~=~\int_{t_0}^{t_f} [u(t)]^T\,v(t)\,dt
\end{equation}
where superscript $T$ stands for transpose of vectors. The entity $x(u,\,t)$ is to be interpreted as an operator that tells us the state vector $x$ for each control vector $u$ and each time $t\in[t_0,\,t_f]\subset\mathbb{R}$ when there are end point conditions which the state variables must satisfy. Working with this operator is, in effect, a supposition that a two point boundary value problem involving the state variables has a solution for each control vector considered. Note that constraints on u are enforced separately \footnote{This definition of $x(u,\,t)$ is precisely that given by Minoux (1986) when analyzing optimal control problems from the point of view of infinite dimensional mathematical programming. Moreover, unless other conditions are satisfied $x(u,\,t)$ is not a solution of the variational inequality considered in (\ref{vi}); rather it should be thought of as a parametric representation of the state vector in terms of the controls. note also that we do not actually have to explicitly solve for $x(u,\,t)$, as is made clear in our subsequent analysis.}, so in working with $x(u,\,t)$ we are not presuming existence of a solution of the variational inequality to be articulated below. 

Furthermore, we assume that every control vector is constrained to lie in a set
$$
U~\subset~\big(\mathcal{L}^2[t_0,\,t_f]\big)^m
$$
where $U$ is defined so as to ensure the terminal conditions imposed on the state variables may be reached from the initial conditions intrinsic to (\ref{eqn1}). Given the operator (\ref{eqn1}), the variational inequality of interest to us takes the following form:
$$
\hbox{find~} u^*\in U~\hbox{such that}
$$
\begin{equation}\label{vi}
\Big<F\big(x(u^*),\,u^*,\,t\big),\,u-u^*\Big>~\geq~0\qquad\forall~u\in U
\end{equation}
where
$$
F:~\big(C^0[t_0,\,t_f]\big)^n\times \big(\mathcal{L}^2[t_0,\,t_f]\big)^m\times \mathbb{R}^1~\longrightarrow~\big(\mathcal{L}^2[t_0,\,t_f]\big)^m
$$
Note that, by virture of the inner product (\ref{eqn4}), we may state the variational inequality (\ref{vi}) as 
$$
\Big<F\big(x(u^*),\,u^*,\,t\big),\,u-u^*\Big>~\equiv~\int_{t_0}^{t_f}\left[F\big(x(u^*),\,u^*,\,t\big)\right]^T(u-u^*)~\geq~0
$$
We refer to (\ref{vi}) as a differential variational inequality with explicit controls and give it the symbolic name $DVI(F,\,f,\,U)$.

\section{Regularity Conditions for Differential Variational Inequalities with Controls}
To analyze (\ref{vi}) we need the following notion of regularity.
\begin{definition}\label{viregularitydef} {\bf (Regularity of $DVI(F,\,f,\,U)$)} We call $DVI(F,\,f,\,U)$ regular if
\begin{enumerate}
\item The state operator $\big(\mathcal{L}^2[t_0,\,t_f]\big)^m \longrightarrow \big(C^0[t_0,\,t_f]\big)^n~:~ u\mapsto x(u,\,\cdot)$ exists and is continuous and G-differentiable with respect to $u$;
\item $\Gamma(x,\,t)$ is continuously differentiable with respect to $x$;
\item $F(x,\,u,\,t)$ is continuous with respect to $x$ and $u$;
\item $f(x,\,u,\,t)$ is convex and continuously differentiable with respect to $x$ and $u$;
\item $U$ is convex and compact; and
\item $x_0\in\mathbb{R}^n$ is known and fixed.
\end{enumerate}
\end{definition}
The motivation for this definition of regularity is to parallel as closely as possible those assumptions needed to analyze traditional optimal control problems from the point of view of infinite dimensional mathematical programming.

To analyze the existence, continuity and differentiability of the state operator, we rely on the following assumption. 

\noindent {\bf (A)} The set $\mathbb{U}\doteq\left\{u(t): ~t\in[t_0,\,t_f],~u\in U\right\} \subset \mathbb{R}^m$ of control values is compact. The function $f=f(x,\,u,\,t)$ is defined and continuous on the space $\mathbb{R}^n\times \mathbb{U}\times \mathbb{R}$, continuously differentiable with respect to $x$, and satisfies
\begin{equation}\label{assumption}
\big|f(x,\,u,\,t)\big|~\leq~C,\qquad \big\|D_x f(x,\,u,\,t)\big\|~\leq~L
\end{equation}
for some constants $C,\,L$ and all $x,\,u,\,t$.

\section{Properties of the state operator}\label{secprop}

\subsection{Existence of the state operator}
The existence of the state operator is an essential issue we must explore. Let us begin with the following theorem to establish the existence of a solution to an ODE.
\begin{theorem}\label{existencethm} {\bf (existence of the solution to ODEs)} Let us consider an initial value problem
\begin{align}
\label{eqn5}
{dx\over dt}&~=~f(x,\,u,\,t)\\
\label{eqn6}
x(t_0)&~=~x_0
\end{align}
for $t\in[t_0,\,t_f]$. Suppose $f(x,\,u,\,t)$ is Lipschitz continuous in $x$ for all $t\in[t_0,\,t_f]$, i.e., the condition 
\begin{equation}\label{lipschitz}
\big|f(x,\,u,\,t)-f(\hat x,\,u,\,t)\big|~\leq~L\big|x-\hat x\big|
\end{equation}
holds for all $x,\,hat x$ and a constant $L\geq 0$. Then the initial value problem (\ref{eqn5})-(\ref{eqn6}) has a unique solution $x(t)$ for $t\in[t_0,\,t_f]$. 
\end{theorem}
\begin{proof}
Following Walter (1988), we begin by writing an equivalent fixed point problem. For each fixed control $u$, the trajectory $x(\cdot,\,u)$ is the fixed point of the transformation $w\mapsto \Phi(u,\,w)$ defined by
\begin{equation}\label{eqn7}
\Phi(u,\,x)(t)~=~x_0+\int_{t_0}^{t}f\big(x(s),\,u(s),\,s\big)\,ds
\end{equation}
The map $\Phi$ is well defined by assumption {\bf (A)}. Our strategy is to show the existence of the fixed point via the Contraction Mapping Theorem (Theorem \ref{fpthm}). We define the norm on $C^0[t_0,\,t_f]$
\begin{equation}\label{normdef}
\|w\|_{\alpha}~\doteq~\max\left\{\big|x(t)\big|e^{-\alpha\,t}:~t\in[t_0,\,t_f]\right\}
\end{equation}
for a constant $\alpha>0$ to be determined later. Then observe
\begin{align*}
\big|\Phi(u,\,x)(t)-\Phi(u,\,\hat x)(t)\big|&~=~\left| \int_{t_0}^t f\big(x(s),\,u(s),\,s)-f\big(\hat x(s),\,u(s),\,s\big)\,ds\right|\\
&~\leq~ \int_{t_0}^t L\big|x(s)-\hat x(s)\big|\,ds\\
&~=~\int_{t_0}^tL\big|x(s)-\hat x(s)\big|e^{-\alpha s} e^{\alpha s}\,ds\\
&~=~L\big\|x-\hat x\big\|_{\alpha}\int_{t_0}^te^{\alpha s}\,ds\\
&~\leq~L\big\|x-\hat x\big\|_{\alpha}{e^{\alpha t}\over \alpha}
\end{align*}
which leads to 
$$
\big|\Phi(u,\,x)(t)-\Phi(u,\,\hat x)(t)\big|e^{-\alpha t}~\leq~{L\over \alpha}\big\|x-\hat x\big\|_{\alpha}
$$
and therefore
$$
\big\|\Phi(u,\,x)(t)-\Phi(u,\,\hat x)(t)\big\|_{\alpha}~\leq~{L\over \alpha}\big\|x-\hat x\big\|_{\alpha}
$$
Choose $\alpha=2 L$ for example, the operator $\Phi(u,\,\cdot)$ is a contraction mapping with constant ${1\over 2}$. By Theorem \ref{fpthm}, a unique $x^*(t)$ satisfying the fixed point problem (\ref{eqn7}) exists, this completes the proof.
\end{proof}

Theorem \ref{existencethm} states that given any $u\in U$, a unique trajectory $x(u,\,\cdot)$ exists for all $t\in[t_0,\,t_f]$. Thus the state operator exists and is well-defined. 

\begin{theorem}\label{continuitythm} {\bf (Continuity of state operator)} 
Let the assumption {\bf (A)} hold. Then the map $u\mapsto x(u,\,\cdot)$ is continuous from $\big(\mathcal{L}^2[t_0,\,t_f]\big)^m$ into $\big(C^0[t_0,\,t_f]\big)^n$. 
\end{theorem}
\begin{proof}
We follow the proof by Bressan and Piccoli (2005). Recall that the trajectory $x(u,\,\cdot)$ is the fixed point of the transformation $w\mapsto \Phi(u,\,w)$ defined by 
$$
\Phi(u,\,x)(t)~=~x_0+\int_{t_0}^tf\big(x(s),\,u(s),\,s\big)\,ds
$$
To show the continuity, let $\left\{u_n\right\}_{n\in \mathbb{N}}$ be a sequence of controls converging to $u$ in the $L^2$-norm. Then $u_n$ also converges to $u$ in the $L^1$-norm, thus for any subsequence $u_{n'}$ we can extract a further subsequence $u_{n''}$ such that $u_{n''}$ converges to $u$ for almost every $t\in[t_0,\,t_f]$. According to the first inequality of (\ref{assumption}) and the Dominated Convergence Theorem, 
$$
\lim_{n''\rightarrow \infty}\int_{t_0}^{t_f}\left|f\big(x(s),\,u(s),\,s\big)-f\big(x(s),\,u_{n''}(s),\,s\big)\right|\,ds~=~0
$$
Since the subsequence $n'$ is arbitrary, we conclude that the above limit holds for the entire sequence $u_n$. Thus 
$$
\big|\Phi(u_n,\,x)(t)-\Phi(u,\,x)(t)\big|~\leq~\int_{t_0}^{t_f}\left| f\big(x(s),\,u_n(s),\,s\big)-f\big(x(s),\,u(s),\,s\big)\right|\,ds\longrightarrow 0
$$
We conclude that $\Phi(u_n,\,x)$ converges to $\Phi(u,\,x)$ uniformly on $[t_0,\,t_f]$. This proves the continuity of $\Phi$ with respect to $u$. 
\end{proof}

Next let us give a differentiability property of the state operator with respect to the control in the Gateaux sense. 

\begin{theorem}\label{diffthm}{\bf (Differentiability of the state operator w.r.t. the control)}
In addition to the assumption {\bf (A)}, assume that $f$ is continuously differentiable in an open neighborhood $\mathbb{V}$ of $\mathbb{U}$. Let $u(\cdot)\in U$ whose corresponding trajectory $x(u,\,\cdot)$ is defined on $[t_0,\,t_f]$. Then for every bounded measurable $\Delta u(\cdot)$ and every $t\in[t_0,\,t_f]$, $x(u,\,\cdot)$ is G-differentiable with respect to $u$. That is, the derivative
\begin{equation}\label{thmder}
\delta x(u,\,\Delta u)~\equiv~\lim_{\varepsilon\rightarrow 0}{x(u+\varepsilon \Delta u,\,t)-x(u,\,t)\over \varepsilon}
\end{equation}
exists for every such $\Delta u$. In particular,
\begin{equation}\label{eqn8}
\delta x(u,\,\Delta u)~=~\int M(t,\,x)D_uf\big(x(u,\,x),\,u(s),\,s\big)\cdot \Delta u(s)\,ds
\end{equation}
where $D_u f$ denotes the matrix of partial derivatives ${\partial f_i\over \partial u_j}$, and $M$ is the matrix fundamental solution for the linearized problem
\begin{equation}\label{eqn9}
\dot v(t)~=~D_x f\big(x(u,\,t),\,u(t),\,t\big)\cdot v(t)
\end{equation}
\end{theorem}
\begin{proof}
We follow the proof by Bressan and Piccoli (2005). Let $z(t)$ be the RHS of (\ref{eqn8}). By Theorem 3.2.6 in Bressan and Piccoli (2005), $z$ is a solution to 
\begin{equation}\label{eqn10}
\dot z(t)~=~A(t)\,z(t)+D_u f\big(x(u,\,t),\,u(t),\,t\big)\cdot\Delta u(t),\qquad z(t_0)~=~0
\end{equation}
where $A(t)=D_xf\big(x(u,\,t),\,u(t),\,t\big)$. If we define 
\begin{align*}
x_{\varepsilon}(t)&~=~x(u+\varepsilon\Delta u,\,t)\\
y_{\varepsilon}(t)&~=~x(u,\,t)+\varepsilon z(t)
\end{align*}
To prove (\ref{eqn8}), it suffices to show that 
\begin{equation}\label{eqn11}
\lim_{\varepsilon\rightarrow 0}\left|{x_{\varepsilon}(t)-y_{\varepsilon}(t)\over \varepsilon}\right|~=~0
\end{equation}
Notice that $x_{\varepsilon}$ is the fixed point of the map $w\mapsto \Phi(u+\varepsilon\Delta u,\,w)$ defined in (\ref{eqn7}), which is contractive with respect to the norm $\|\cdot\|_{\alpha}$ as in (\ref{normdef}). Choosing $\alpha=2L$ together with estimate (\ref{fpthmeqn3}) yields
$$
{1\over \varepsilon}\left\|x_{\varepsilon}-y_{\varepsilon}\right\|_{2L}~\leq~{2\over\varepsilon}\left\|\Phi(u+\varepsilon\Delta u)-y_{\varepsilon}\right\|
$$
To prove (\ref{eqn11}), it suffices to show that 
\begin{equation}\label{eqn12}
\lim_{\varepsilon\rightarrow 0}\left(\sup_{t\in[t_0,\,t_f]}{1\over \varepsilon}\left|x_0+\int_{t_0}^{t_f}f\big(y_{\varepsilon}(s),\,u+\varepsilon\Delta u,\,s\big)\,ds-y_{\varepsilon}(t)\right|\right)~=~0
\end{equation}
Recalling the definition of $y_{\varepsilon}(\cdot)$, we obtain
\begin{align*}
&{1\over\varepsilon} \left| x_0+\int_{t_0}^tf\big(y_{\varepsilon}(s),\,u+\varepsilon\Delta u,\,s\big)\,ds-y_{\varepsilon}(t)\right|\\
=~&{1\over\varepsilon}\left|x_0+\int_{t_0}^tf\big(x(u,\,s)+\varepsilon z(s),\,u+\varepsilon\Delta u,\,s\big)\,ds -x(u,\,t)-\varepsilon z(t)\right|\\
=~&{1\over \varepsilon}\left|\left\{x_0+\int_{t_0}^tf\big(x(u,\,s),\,u(s),\,s\big)\,ds+\int_{t_0}^tD_xf\big(x(u,\,s),\,u(s),\,s\big)\cdot \varepsilon z(s)\,ds\right.\right.\\
&+\left. \int_{t_0}^tD_uf\big(x(u,\,s),\,u(s),\,s\big)\cdot\varepsilon\Delta u(s)\,ds-\varepsilon z(t)\right\}\\
&+\int_{t_0}^t\int_0^1\left[D_xf\big(x(u,\,s)+\sigma\varepsilon z(s),\,u(s)+\sigma\varepsilon\Delta u(s),\,s\big)-D_xf\big(x(u,\,s),\,u(s),\,s\big)\right]\cdot\varepsilon z(s)\,d\sigma ds\\
&+\left.\int_{t_0}^t\int_0^1\left[D_uf\big(x(u,\,s)+\sigma\varepsilon z(s),\,u(s)+\sigma\varepsilon\Delta u(s),\,s\big)-D_uf\big(x(u,\,s),\,u(s),\,s\big)\right]\cdot\varepsilon z(s)\,d\sigma ds\right|\\
\leq~&\int_{t_0}^{t_f}\int_0^1 \left\|D_xf\big(x(u,\,s)+\sigma\varepsilon z(s),\,u(s)+\sigma\varepsilon\Delta u(s),\,s\big)-D_xf\big(x(u,\,s),\,u(s),\,s\big)\right\|\cdot |z(s)|\,d\sigma ds\\
&+\int_{t_0}^{t_f}\int_0^1\left\|D_u f\big(x(u,\,s)+\sigma\varepsilon z(s),\,u(s)+\sigma\varepsilon\Delta u(s),\,s\big)-D_uf\big(x(u,\,s),\,u(s),\,s\big)\right\|\cdot|\Delta u(s)|\,d\sigma ds
\end{align*}
By the Dominated Convergence Theorem, the above right hand side converges to zero as $\varepsilon\rightarrow 0$, proving (\ref{eqn12}) and hence (\ref{eqn11}).
\end{proof}

\section{Application to Dynamic User Equilibrium} 

\subsection{Review of the differential variational inequality formulation of DUE}
We will assume the following planning time horizon for the time being
$$
[t_0,\,t_f]\subset \mathbb{R}_+
$$
The most crucial component of the DUE model is the path delay operator, which provides the time delay on any path $p$ per unit flow departing from the origin of that path. The delay operator is denoted by 
$$
D_p(t,\,h)\qquad \forall ~p\in\mathcal{P}
$$
where $\mathcal{P}$ is the set of all paths employed by network users, $t$ denotes the departure time, $h$ is a vector of departure rates. Throughout the rest of the paper, we stipulate that 
$$
h\in \Big(\mathcal{L}_+^2([t_0,\,t_f]\Big)^{|\mathcal{P}|}
$$
where $\Big(\mathcal{L}_+^2([t_0,\,t_f])\Big)^{|\mathcal{P}|}$ denotes the non-negative cone of the $|\mathcal{P}|$-fold product of the Hilbert space $\mathcal{L}^2([t_0,\,t_f])$ of square-integrable functions on the compact interval $[t_0,\,t_f]$. The inner product of the Hilbert space $\Big(\mathcal{L}^2([t_0,\,t_f])\Big)^{|\mathcal{P}|}$ is defined as 
\begin{equation}\label{nnorm}
\big<u,\,v\big>~\doteq~\int_{t_0}^{t_f} (u(s))^T\,v(s)\,ds
\end{equation}
where the superscript $T$ denotes transpose of vectors. This inner product induces the norm 
\begin{equation}\label{l2norm}
\big\|u\big\|_{\mathcal{L}^2}~\doteq~\big<u,\,u\big>^{1/2}
\end{equation}

 Next, for each $p\in\mathcal{P}$, we define the effective unit path delay operator $\Psi_p: [t_0,\,t_f]\times \Big(\mathcal{L}_+^2([t_0,\,t_f])\Big)^{|\mathcal{P}|}\rightarrow \mathbb{R}$ via
\begin{equation}\label{phidef}
\Psi_p(t,\,h)~=~D_p(t,\,h)+\mathcal{F}\Big(t+D_p(t,\,h)-T_A\Big)
\end{equation}
where $\mathcal{F}(\cdot)$ is the penalty for early or later arrival relative to the target arrival time $T_A$. We interpret $\Psi_p(t,\,h)$ as the perceived travel cost of driver starting at time $t$ on path $p$ under travel conditions $h$. Presently, our only assumption on such costs is that for each $h\in\Big(\mathcal{L}_+^2([t_0,\,t_f])\Big)^{|\mathcal{P}|}$, the function $\Psi(\cdot,\,h): [t_0,\,t_f]\rightarrow \mathbb{R}$ is measurable. This minimal assumption was used for a measure theory-based argument in \cite{Friesz1993}.

%Later in this article, we shall establish more properties of this operator such as continuity on a Hilbert space under minor condition on the function $\mathcal{F}(\cdot)$. This property is crucial for applying the result in \cite{Browder}. See XXXX (PRELIMINARY).

To support the development of a dynamic network user equilibrium model in the following, we introduce some additional constraints, namely, the flow conservation constraints, 
\begin{equation}\label{flowcons}
\sum_{p\in\mathcal{P}_{ij}}\int_{t_0}^{t_f} h_p(t)\,dt~=~Q_{ij}\qquad\forall ~(i,\,j)\in\mathcal{W}
\end{equation}
where $\mathcal{P}_{ij}$ is the set of all paths that connect origin-destination (O-D) pair $(i,\,j)$. $\mathcal{W}$ is the set of all O-D pairs. In addition, $Q_{ij}$ is the fixed travel demand for O-D pair $(i,\,j)$.  Using the notation and concepts we have mentioned, the feasible region for DUE when the effective delay operator $\Psi(\cdot,\,\cdot)$ is given is 
\begin{equation}\label{feasible}
\Lambda~=~\left\{h\in\Big(\mathcal{L}_+^2([t_0,\,t_f])\Big)^{|\mathcal{P}|}: \quad \sum_{p\in\mathcal{P}_{ij}}\int_{t_0}^{t_f} h_p(t)\,dt~=~Q_{ij}\quad\forall ~(i,\,j)\in\mathcal{W}\right\}
\end{equation}

The following definition of dynamic user equilibrium was first articulated in \cite{Friesz1993}:
\begin{definition}\label{duedef}{\bf (Dynamic user equilibrium)}. A vector of departure rates (path
flows) $h^{\ast }\in \Lambda $ is a dynamic user equilibrium if%
\begin{equation}\label{defdue}
h_{p}^{\ast }\left( t\right) >0,p\in P_{ij}\Longrightarrow \Psi _{p}\left[
t,h^{\ast }\left( t\right) \right] =v_{ij}
\end{equation}
We denote this equilibrium by $DUE\left( \Psi ,\Lambda ,\left[
t_{0},t_{f}\right] \right) $.
\end{definition}

Using measure theoretic argument, Friesz et al. (1993) established that a dynamic user equilibrium is equivalent to the following variational inequality under suitable regularity conditions:
\begin{equation}\label{duevi}
\left. 
\begin{array}{c}
\text{find }h^*\in \Lambda \text{ such that} \\ 
\displaystyle \sum_{p\in \mathcal{P}} \int_{t_0}^{t_{f}}\Psi _{p}(t,h^*)(h_{p}-h_{p}^{\ast })dt\geq 0 \\ 
\forall h\in \Lambda %
\end{array}%
\right\} VI \Big(\Psi ,\Lambda ,[ t_{0},t_{f}] \Big)  
\end{equation}

It has been reported in Friesz et al. (2010) that (\ref{duevi}) is equivalent to a differential variational inequality. This is most easily seen by noting that the flow conservation constraints may be re-stated as

\begin{equation}\label{tpbdy}
\left.
\begin{array}{l}
\displaystyle {d Y_{ij}\over dt}~=~\sum_{p\in\mathcal{P}_{ij}}h_p(t)\\
Y_{ij}(t_0)~=~0\\
Y_{ij}(t_f)~=~Q_{ij}\\
 \forall\,(i,\,j)\in\mathcal{W}
\end{array}\right\}
\end{equation}
which is recognized as a two boundary value problem. As a consequence (\ref{duevi}) may be re-written as the following differential variational inequality (DVI):
\begin{equation}\label{duedvi}
\left. 
\begin{array}{c}
\text{find }h^*\in \Lambda_1 \text{ such that} \\ 
\displaystyle \sum_{p\in \mathcal{P}} \int_{t_0}^{t_{f}}\Psi _{p}(t,h^*)(h_{p}-h_{p}^{\ast })dt\geq 0 \\ 
\forall h\in \Lambda_1 %
\end{array}%
\right\} DVI \Big(\Psi ,\Lambda_1 ,[ t_{0},t_{f}] \Big)  
\end{equation}
where 
\begin{align*}
\Lambda_1~=~&\left\{h\geq 0:~{d Y_{ij}\over dt}=\sum_{p\in\mathcal{P}_{ij}}h_p(t),~Y_{ij}(t_0)=0,~Y_{ij}(t_f)=Q_{ij},\quad \forall\,(i,\,j)\in\mathcal{W}\right\}\\
&\subset \Big(\mathcal{L}_+^2([t_0,\,t_f])\Big)^{|\mathcal{P}|}
\end{align*}

\subsection{The state operator}
In this section, we will apply the results established in Section \ref{secprop} to the DVI (\ref{duedvi}), and obtain properties of the state operator in $DVI\big(\Psi,\,\Lambda_1,\,[t_0,\,t_f]\big)$ such as existence and regularity. These properties of the state operator will allow us to further analyze  $DVI\big(\Psi,\,\Lambda_1,\,[t_0,\,t_f]\big)$. 

We begin by identifying the control set $U=\Lambda_1$, and control $h\in U$. The state variable becomes $Y=\big(Y_{ij}(h,\,t)\big)\in \big(C^0[t_0,\,t_f]\big)^{|\mathcal{W}|}$. Then the abstract definition of state operator (\ref{eqn1})-(\ref{eqn3}) can be instantiated as following 
\begin{multline}\label{eqn13}
Y(h,\,t)~=~\left\{Y:~{d Y\over dt}=f(Y,\,h,\,t),~Y(t_0)=Y_0,~\Gamma\left[Y(t_f),\,t_f\right]=0\quad \forall \,(i,\,j)\in\mathcal{W}\right\}\\\in\Big(C^0[t_0,\,t_f]\Big)^{|\mathcal{W}|}
\end{multline}
where
\begin{align}
\label{eqn14}
&Y_0=0~\in~\mathbb{R}^{|\mathcal{W}|}\\
\label{eqn15}
&f: \mathbb{R}_+^{|\mathcal{W}|}\times \mathbb{R}_+^{|\mathcal{P}|}\times \mathbb{R}^1~\longrightarrow ~\mathbb{R}_+^{|\mathcal{P}|},\quad \big[f(Y,\,h,\,t)\big]_{ij}~=~\sum_{p\in\mathcal{P}_{ij}}h_p,\quad \forall \,(i,\,j)\in\mathcal{W}\\
\label{eqn16}
&\Gamma: \mathbb{R}_+^{|\mathcal{W}|}\times\mathbb{R}^1~\longrightarrow~\mathbb{R}^{|\mathcal{W}|},\qquad \big[\Gamma[Y,\,t]\big]_{ij}~=~Y_{ij}-Q_{ij},\qquad \forall\,(i,\,j)\in\mathcal{W}
\end{align}

The following results are straightforward.

\begin{proposition}\label{simpleprop1}{\bf (Existence and continuity of the state operator)}
The state operator $h\mapsto Y(h,\,\cdot)$ is well-defined and is continuous from $\Lambda_1$ into $\big(C^0[t_0,\,t_f]\big)^{|\mathcal{P}|}$.
\end{proposition}

\begin{proposition}\label{simpleprop2}{\bf (G-differentiability of the state operator)}
Let $h(\cdot)\in\Lambda_1$ be a control whose corresponding solution $Y(h,\,\cdot)$ is defined on $[t_0,\,t_f]$. Then for every bounded measurable $\Delta h(\cdot)$ and every $t\in[t_0,\,t_f]$, the map $\varepsilon \mapsto Y(h+\varepsilon\Delta h,\,t)$ is differentiable. In particular, if the derivative is defined as 
$$
\delta Y(h,\,\Delta h)~\doteq~\lim_{\varepsilon\rightarrow 0}{Y(h+\varepsilon\Delta h,\,t)-Y(h,\,t)\over \varepsilon}
$$
Then
$$
\delta Y(h,\,\Delta h)~\in~\mathbb{R}^{|\mathcal{W}|},\qquad \big[\delta Y(h,\,\Delta h)\big]_{ij}~=~\sum_{p\in\mathcal{P}_{ij}}\int_{t_0}^t \Delta h_p(s)\,ds,\qquad\forall\,(i,\,j)\in\mathcal{W}
$$
\end{proposition}

\end{document}